\documentclass[11pt,reqno]{amsart}
\usepackage{fourier}
\usepackage{fullpage}
\usepackage{mathrsfs,amssymb,graphicx,verbatim,amsmath,amsfonts}
\usepackage{paralist}
\usepackage[breaklinks,pdfstartview=FitH]{hyperref}
\usepackage{upgreek}
\usepackage[mathscr]{euscript}

\usepackage{xcolor}

\makeatletter
\def\resetMathstrut@{%
  \setbox\z@\hbox{%
    \mathchardef\@tempa\mathcode`\(\relax
    \def\@tempb##1"##2##3{\the\textfont"##3\char"}%
    \expandafter\@tempb\meaning\@tempa \relax
  }%
  \ht\Mathstrutbox@1.2\ht\z@ \dp\Mathstrutbox@1.2\dp\z@
}
\makeatother
\renewcommand{\le}{\leqslant}
\renewcommand{\ge}{\geqslant}

\renewcommand{\setminus}{\smallsetminus}
\renewcommand{\gamma}{\upgamma}
\newcommand{\ud}[0]{\,\mathrm{d}}

\usepackage{dutchcal}

\renewcommand{\det}{\mathrm{det}}

\newcommand{\proj}{\mathsf{Proj}}

\newcommand{\n}{\{1,\ldots,n\}}
\newcommand{\m}{\{1,\ldots,m\}}

\newcommand{\X}{\mathbf X}

\newcommand{\sfA}{\mathsf{A}}
\newcommand{\sfB}{\mathsf{B}}
\newcommand{\sfC}{\mathsf{C}}

\renewcommand{\d}{\delta}

\newcommand{\e}{\varepsilon}
\newcommand{\R}{\mathbb R}
\newcommand{\Q}{\mathbb Q}
\newcommand{\1}{\mathbf 1}
\newcommand{\0}{\mathbf 0}

\newtheorem{theorem}{Theorem}
\newtheorem{lemma}[theorem]{Lemma}

\newtheorem{observation}[theorem]{Observation}

\newtheorem{remark}[theorem]{Remark}

\newtheorem{notation}[theorem]{Notation}

\newtheorem{question}[theorem]{Question}

\renewcommand{\subset}{\subseteq}

\DeclareMathOperator{\trace}{\mathrm{Trace}}
\newcommand{\E}{\mathbb{ E}}

\newcommand{\N}{\mathbb N}
\newcommand{\Z}{\mathbb Z}

\newcommand{\eqdef}{\stackrel{\mathrm{def}}{=}}

\renewcommand{\emptyset}{\varnothing}

\newcommand{\M}{\mathsf{M}}
\newcommand{\GL}{\mathsf{GL}}

\renewcommand{\Pr}{\mathrm{Prob}}

\newcommand{\vol}{\mathrm{vol}}

\newcommand{\sfP}{\mathsf{P}}

\newcommand{\spn}{\mathrm{span}}
\begin{document}

\title{An integer parallelotope with small surface area}\thanks{A.N.~was supported by NSF grant DMS-2054875, BSF grant  201822, and a Simons Investigator award. O.R.~was supported by NSF grant CCF-1320188 and a Simons Investigator award.}

\author{Assaf Naor}
\address{Mathematics Department\\ Princeton University\\ Fine Hall, Washington Road, Princeton, NJ 08544-1000, USA}
\email{naor@math.princeton.edu}

\author{Oded Regev}
\address{Department of Computer Science\\ Courant Institute of Mathematical Sciences, New York University\\ 251 Mercer Street, New York, NY 10012, USA}
\email{regev@cims.nyu.edu}

\maketitle

\begin{abstract} We prove that for any $n\in \N$ there is a convex body $K\subset \R^n$ whose surface area is at most $n^{\frac12+o(1)}$, yet the translates of $K$ by the integer lattice $\Z^n$ tile $\R^n$.

\end{abstract}

\section{Introduction}

Given $n\in \N$ and a lattice $\Lambda\subset \R^n$, a convex body $K\subset \R^n$ is called a $\Lambda$-parallelotope (e.g.,~\cite{Eng93})  if the translates of $K$ by elements of $\Lambda$ tile $\R^n$, i.e., $\R^n=\Lambda+K=\bigcup_{x\in \Lambda} (x+K)$, and the interior of $(x+K)\cap (y+K)$ is empty for every distinct $x,y\in \Lambda$. One calls $K$ a parallelotope (parallelogon if $n=2$ and parallelohedron if $n=3$; some of the literature calls a parallelotope in $\R^n$ and $n$-dimensional parallellohedron; e.g.,~\cite{Ale05,Dol12}) if it is a $\Lambda$-parallelotope for some lattice $\Lambda\subset \R^n$. We call a $\Z^n$-parallelotope an integer parallelotope.

The hypercube $[-\frac12,\frac12]^n$ is an integer parallelotope whose surface area equals $2n$. By~\cite[Corollary~A.2]{FKO07}, for every $n\in \N$ there exists an integer parallelotope $K\subset \R^n$ whose surface area is smaller than $2n$ by a universal constant factor. Specifically, the surface area of the integer parallelotope $K$ that was considered in~\cite{FKO07} satisfies $\vol_{n-1}(\partial K)\le \sigma (n+O(n^{2/3}))$, where  $\sigma=2\sum_{s=1}^\infty (s/e)^s/(s^{3/2}s!)\le 1.23721$. To the best of our knowledge, this is the previously best known upper bound on the smallest possible surface area of an integer parallelotope. The main result of the present work is the following theorem:

\begin{theorem}\label{thm:main} For every $n\in \N$ there exists an integer parallelotope whose surface area is $n^{\frac12+o(1)}$.
\end{theorem}

Because the covolume of $\Z^n$ is $1$, the volume of any integer parallelotope $K\subset \R^n$ satisfies $\vol_n(K)=1$. Consequently, by the isoperimetric inequality we have\footnote{We use the following conventions for asymptotic notation, in addition to the usual $O(\cdot),o(\cdot),\Omega(\cdot),\Theta(\cdot)$ notation. For $a,b>0$, by writing
$a\lesssim b$ or $b\gtrsim a$ we mean that $a\le Cb$ for a
universal constant $C>0$, and $a\asymp b$
stands for $(a\lesssim b) \wedge  (b\lesssim a)$. If we need to allow for dependence on parameters, we indicate it by subscripts. For example, in the presence of an auxiliary parameter $\e$, the notation $a\lesssim_\e b$ means that $a\le C(\e)b$, where $C(\e)>0$ may depend only on $\e$, and analogously for  $a\gtrsim_\e b$ and $a\asymp_\e b$.}
\begin{equation}\label{eq:use isoperimetric}
\vol_{n-1}(\partial K)\ge \frac{\vol_{n-1}(S^{n-1})}{\vol_n(B^n)^{\frac{n-1}{n}}} \asymp \sqrt{n},
\end{equation}
where $B^n\eqdef \{(x_1,\ldots,x_n)\in \R^n:\ x_1^2+\cdots+x_n^2\le 1\}$ denotes the Euclidean ball and $S^{n-1}\eqdef \partial B^n$.

Thanks to~\eqref{eq:use isoperimetric}, Theorem~\ref{thm:main} is optimal up to the implicit lower order factor. It remains open to determine whether this lower-order factor could be removed altogether, namely to answer the following question:

\begin{question}\label{Q:sharp} For every $n\in \N$, does there exist  an integer parallelotope $K\subset \R^n$ with $\vol_{n-1}(\partial K)\asymp \sqrt{n}$?
\end{question}

Question~\ref{Q:sharp} goes back to~\cite{Kel94}, though such early investigations were (naturally, from the perspective of crystallography) focused on $n=3$ and asked for the exact value of the smallest possible surface area of a parallelohedron; see Conjecture~7.5 in~\cite{Bez06} and the  historical discussion in the paragraph that precedes it. The corresponding question about precisely determining the minimum  perimeter when  $n=2$ was answered in~\cite{Cho89} (its solution for general parallelogons rather than integer parallelogons is due to~\cite{Fej43}; see also~\cite{Hal01}, which treats tiles that need not be convex). Finding the exact minimum when $n=3$ remains open; we will not review the substantial literature on this topic, referring instead to the monograph~\cite{TW08} (see also~\cite{Lan22} for an exact solution of a different isoperimetric-type question for parallelohedra).

The higher dimensional asymptotic nature of Question~\ref{Q:sharp} differs from the search for exact minimizers in lower dimensions on which the literature has focused, but it is a natural outgrowth of it and it stands to reason that it was considered by researchers who worked on this topic over the past  centuries. Nevertheless, we do not know of a published source that mentions Question~\ref{Q:sharp}  prior to the more recent interest in this topic that arose due to its connection  to theoretical computer science that was found in~\cite{FKO07} and were pursued in~\cite{Raz11,KORW08,AK09,KROW12,BM21}; specifically, Question~\ref{Q:sharp} appears in~\cite[Section~6]{BM21}.


In~\cite{KORW08} it was proved that Question~\ref{Q:sharp} has a positive answer if one drops the requirement that the tiling set is convex, i.e., by~\cite[Theorem~1.1]{KORW08} for every $n\in \N$ there is a compact set $\Omega\subset \R^n$ such that $\R^n=\Z^n+\Omega$,  the interior of $(x+\Omega)\cap (y+\Omega)$ is empty for every distinct $x,y\in \Z^n$, and $\vol_{n-1}(\partial \Omega)\lesssim \sqrt{n}$; see also the proof of this result that was found in~\cite{AK09}. The lack of convexity of $\Omega$ is irrelevant for the applications to computational complexity that were found in~\cite{FKO07}. The proofs in~\cite{KORW08,AK09} produce a set $\Omega$ that is decidedly non-convex. Our proof of Theorem~\ref{thm:main}  proceeds via an entirely different route and provides a paralletotope whose surface area comes close to the  guarantee of~\cite{KORW08} (prior to~\cite{KORW08}, the best known upper bound on the smallest possible surface area of a compact $\Z^n$-tiling set was the aforementioned $1.23721n$ of~\cite{FKO07}).

While it could be tempting to view the existence of the aforementioned compact set $\Omega$ as evidence for the availability of an integer parallelotope with comparable surface area, this is a tenuous hope because the convexity requirement from a parallelotope imposes severe restrictions. In particular, by~\cite{Min97} for every $n\in \N$ there are only finitely many combinatorial types of parallelotopes in $\R^n$.\footnote{Thus, just for the sake concreteness (not important for the present purposes): Since antiquity it was known that there are $2$ types of parallelogons; by~\cite{Fed53} there are 5  types of parallelohedra; by~\cite{Del29,Sht73} there are 52 types of $4$-dimensional parallelotopes.} In fact, by combining~\cite[Section~6]{Dol09} with~\cite{Min97,Ven54} we see that $K\subset \R^n$ is a parallelotope if and only if $K$ is a centrally symmetric polytope, all of the $(n-1)$-dimensional faces of $K$ are centrally symmetric, and the orthogonal projection of $K$ along any of its $(n-2)$-dimensional faces is either a parallelogram or a centrally symmetric hexagon.

Of course, Theorem~\ref{thm:main} {\em must} produce such a constrained polytope. To understand how this is achieved, it is first important to stress that this becomes a  straightforward task if one only asks for a parallelotope with small surface area rather than for an {\em integer} parallelotope with small surface area. Namely, it follows easily from the literature that for every $n\in \N$ there exist a rank $n$ lattice $\Lambda\subset \R^n$ whose  covolume is $1$ and a $\Lambda$-parallelotope $K\subset \R^n$ that satisfies $\vol_{n-1}(\partial K)\lesssim \sqrt{n}$. Indeed, by~\cite{Rog50} there is a rank $n$ lattice $\Lambda\subset \R^n$ of covolume $1$ whose packing radius is at least $c\sqrt{n}$, where $c>0$ is a universal constant. Let $K$ be the Voronoi cell of $\Lambda$, namely $K$ consists of the points in $\R^n$ whose (Euclidean) distance to any point of $\Lambda$ is not less than their distance to the origin. Then, $K$ is a $\Lambda$-parallelotope, $\vol_n(K)=1$ since the covolume of $\Lambda$ is $1$, and $K\supseteq c\sqrt{n}B^n$ since the packing radius of $\Lambda$ is at least $c\sqrt{n}$. Consequently,  the surface area of $K$ is at most $c^{-1}\sqrt{n}$ by the following simple lemma that we will use multiple times in the proof of Theorem~\ref{thm:main}:

\begin{lemma}\label{lem:inradius} Fix $n\in \N$ and $R>0$. Suppose that a convex body $K\subset \R^n$ satisfies $K\supseteq RB^n$. Then,
$$
\frac{\vol_{n-1}(\partial K)}{\vol_n(K)}\le \frac{n}{R}.
$$
\end{lemma}
\noindent Lemma~\ref{lem:inradius} is known (e.g.,~\cite[Lemma~2.1]{GKV18}); for completeness we will present  its short proof in Section~\ref{sec:proof}.

Even though the packing radius of $\Z^n$ is small, the above observation drives our inductive proof of Theorem~\ref{thm:main}, which proceeds along the following lines. Fix $m\in \{1,\ldots, n-1\}$ and let $V$ be an $m$-dimensional subspace of $\R^n$. If  the lattice $V^\perp \cap \Z^n$ has rank $n-m$ and its packing radius is large, then Lemma~\ref{lem:inradius} yields a meaningful upper bound on the $(n-m-1)$-dimensional  volume of the boundary of the Voronoi cell of $V^\perp \cap \Z^n$. We could then consider the lattice $\Lambda\subset V$ which is the orthogonal projection of $\Z^n$ onto $V$, and inductively obtain a $\Lambda$-parallelotope (residing within $V$) for which the $(m-1)$-dimensional volume of its boundary is small. By considering the product (with respect to the identification of $\R^n$ with $V^\perp\times V$) of the two convex bodies thus obtained, we could hope to get the desired integer parallelotope.

There are obvious obstructions to this plan. The subspace $V$ must be chosen so that the lattice $V^\perp \cap \Z^n$ is sufficiently rich yet it contains no short nonzero vectors. Furthermore, the orthogonal projection $\Lambda$ of $\Z^n$ onto $V$ is not $\Z^m$, so we must assume a stronger inductive hypothesis and also apply a suitable ``correction'' to $\Lambda$ so as to be able to continue the induction. It turns out that there is tension between how large the packing radius of $V^\perp \cap \Z^n$   could be, the loss that we incur due to the aforementioned correction, and the total cost of iteratively applying the procedure that we sketched above. Upon balancing these constraints, we will see that the best choice for the dimension $m$ of $V$ is $m=n\exp(-\Theta(\sqrt{\log n}))$. The rest of the ensuing text will present the details of the implementation of this strategy.

\section{Proof of Theorem~\ref{thm:main}}\label{sec:proof}

Below, for each $n\in \N$ the normed space $\ell_2^n=(\R^n,\|\cdot\|_{\ell_2^n})$ will denote the standard Euclidean space, i.e.,
$$
\forall x=(x_1,\ldots,x_n)\in \R^n,\qquad \|x\|_{\ell_2^n}\eqdef \sqrt{x_1^2+\cdots+x_n^2}.
$$
The standard scalar product of $x,y\in \R^n$ will be denoted $\langle x,y\rangle \eqdef x_1y_1+\cdots+x_ny_n$. The coordinate basis of $\R^n$ will be denoted $e_1,\ldots,e_n$, i.e., for each $i\in \n$ the $i$th entry of $e_i$ is $1$ and the rest of the coordinates of $e_i$ vanish. We will denote the origin of $\R^n$ by $\0=(0,\ldots,0)$.  For $0<s\le n$, the $s$-dimensional Hausdorff measure on $\R^n$ that is induced by the $\ell_2^n$ metric will be denoted by $\vol_s(\cdot)$. In particular, if $K\subset \R^n$ is a convex body (compact and with nonempty interior), then the following identity holds (see, e.g.,~\cite{KR97}):
\begin{equation}\label{eq:surface formula limit}
\vol_{n-1}(\partial K)=\lim_{\d\to 0^+} \frac{\vol_n(K+\d B^n)-\vol_n(K)}{\d}.
\end{equation}

If $V$ is a subspace of $\R^n$, then its orthogonal complement (with respect to the $\ell_2^n$ Euclidean structure) will be denoted $V^\perp$ and the orthogonal projection from $\R^n$ onto $V$ will be denoted $\proj_V$. When treating a subset $\Omega$ of $V$ we will slightly abuse notation/terminology by letting  $\partial \Omega$ be the boundary of $\Omega$ within $V$, and similarly when we will discuss the interior of $\Omega$ we will mean its interior within $V$. This convention results in suitable interpretations of when $K\subset V$ is a convex body or a parallelohedron (with respect to a lattice of $V$). The variant of~\eqref{eq:surface formula limit} for a convex body $K\subset V$ becomes
\begin{equation}\label{eq:area of boundary in subspace}
\vol_{\dim(V)-1}(\partial K)=\lim_{\d\to 0^+} \frac{\vol_{\dim(V)}\big(K+\d (V\cap B^n)\big)-\vol_{\dim(V)}(K)}{\d}.
\end{equation}

\begin{proof}[Proof of Lemma~\ref{lem:inradius}] Since $K\supseteq R B^n$, for every $\d>0$ we have
\begin{equation}\label{eq:use convexity}
K+\d B^n\subset K+\frac{\d}{R} K=\Big(1+\frac{\d}{R}\Big) \Big( \frac{R}{R+\d}K+\frac{\d}{R+\d}K\Big) = \Big(1+\frac{\d}{R}\Big)K,
\end{equation}
where the last step of~\eqref{eq:use convexity} uses the fact that $K$ is convex. Consequently,
\begin{equation*}
\vol_{n-1}(\partial K)\stackrel{\eqref{eq:surface formula limit}}{=}\lim_{\d\to 0^+} \frac{\vol_n(K+\d B^n)-\vol_n(K)}{\d}\stackrel{\eqref{eq:use convexity}}{\le} \lim_{\d\to 0^+} \frac{\big(1+\frac{\d}{R}\big)^n-1}{\d}\vol_n(K)=\frac{n}{R}\vol_n(K).\tag*{\qedhere}
\end{equation*}
\end{proof}

The sequence $\{Q(n)\}_{n=1}^\infty$ that we introduce in the following definition will play an important role in the ensuing reasoning: 

\begin{notation}\label{notation:Q} For each $n\in \N$ let $Q(n)$ be the infimum over those $Q\ge 0$ such that for every lattice $\Lambda\subset \Z^n$ of rank $n$ there exists a $\Lambda$-parallelotope $K\subset \R^n$ that satisfies
\begin{equation}\label{eq:def Q}
\frac{\vol_{n-1}(\partial K)}{\vol_n(K)}\le Q.
\end{equation}
\end{notation}

As $\vol_n(K)=1$ for any integer parallelotope $K\subset \R^n$, Theorem~\ref{thm:main} is a special case of the following result: 

\begin{theorem}\label{thm:main Q version} There exists a universal constant $C\ge 1$ such that $Q(n)\lesssim \sqrt{n} e^{C\sqrt{\log n}}$ for every $n\in \N$ .
\end{theorem}

The following key lemma is the inductive step in the ensuing proof of Theorem~\ref{thm:main Q version} by induction on $n$:

\begin{lemma}\label{lem:inductive step} Fix $m,n,s\in \N$ with $s\le m\le n$. Suppose that $\sfB\in \M_{m\times n}(\Z)$ is an $m$-by-$n$ matrix all of whose entries are integers such that $\sfB$ has rank $m$ and any $s$ of the columns of $\sfB$ are linearly independent. Then,
\begin{equation*}\label{eq:recursive inequality}
Q(n)\le \frac{2(n-m)}{\sqrt{s}}+Q(m)\|\sfB\|_{\ell_2^n\to \ell_2^m},
\end{equation*}
where  $\|\cdot\|_{\ell_2^n\to \ell_2^m}$ denotes the operator norm from $\ell_2^n$ to $\ell_2^m$.
\end{lemma}

The fact that Theorem~\ref{thm:main Q version} treats any sublattice of $\Z^n$ of full rank (recall how $Q(n)$ is defined), even though in Theorem~\ref{thm:main} we are interested  only in $\Z^n$ itself, provides a strengthening of the inductive hypothesis that makes it possible for our proof of Lemma~\ref{lem:inductive step} to go through. If $\Lambda$ is an arbitrary full rank sublattice of $\Z^n$, then a $\Lambda$-parallelotope $K\subset \R^n$ need no longer satisfy $\vol_n(K)=1$, so the inductive hypothesis must incorporate the value of $\vol_n(K)$, which is the reason why we consider the quantity $\vol_{n-1}(\partial K)/\vol_n(K)$ in~\eqref{eq:def Q}. Observe that this quantity is not scale-invariant, so it might seem somewhat unnatural to study it, but it is well-suited to the aforementioned induction thanks to the following simple lemma:

\begin{lemma}\label{lem:properties of ratio} Fix $m,n\in \N$ and  an $m$-dimensional subspace $V$ of $\R^n$. Let $O\subset V^\perp$ be an open subset of  $V^\perp$ and let $G\subset V$ be an open subset of $V$. Then, for $\Omega= O+G$ we have
\begin{equation}\label{eq:additivity}
\frac{\vol_{n-1}(\partial \Omega)}{\vol_n(\Omega)}=\frac{\vol_{n-m-1}(\partial O)}{\vol_{n-m}(O)}+\frac{\vol_{m-1}(\partial G)}{\vol_m(G)}.
\end{equation}
Furthermore, if $T:\R^m\to V$ is a linear isomorphism and $K\subset \R^m$ is a convex body, then 
\begin{equation}\label{eq:cheeger ratio for linear image}
\frac{\vol_{m-1}(\partial TK)}{\vol_{m}(TK)}\le  \frac{\vol_{m-1}(\partial K)}{\vol_{m}(K)}\|T^{-1}\|_{(V,\|\cdot\|_{\ell_2^n})\to \ell_2^m},
\end{equation}
where $\|\cdot\|_{(V,\|\cdot\|_{\ell_2^n})\to \ell_2^m}$ is the operator norm  from $V$, equipped with the norm inherited from $\ell_2^n$, to $\ell_2^m$.
\end{lemma}

\begin{proof} For~\eqref{eq:additivity}, note that since $O\perp G$ we have $\vol_n(\Omega)=\vol_{n-m}(O)\vol_m(G)$, and   $\partial \Omega=(\partial O+G)\cup (O+\partial G)$ where $\vol_{n-1} ((\partial O+G)\cap (O+\partial G))=0$, so $\vol_{n-1}(\partial \Omega)=\vol_{n-m-1}(\partial O)\vol_m(G)+ \vol_{n-m}(O)\vol_{m-1}(\partial G)$.

For~\eqref{eq:cheeger ratio for linear image}, denote $\rho= \|T^{-1}\|_{(V,\|\cdot\|_{\ell_2^n})\to \ell_2^m}$,  so that $T^{-1} (V\cap B^n)\subset \rho B^m$. Consequently, 
$$
\forall \d\in \R,\qquad TK+\d(V\cap B^n)=T\big(K+\d T^{-1}(V\cap B^n)\big)\subset T(K+\d\rho B^m).
$$
By combining this inclusion with~\eqref{eq:area of boundary in subspace}, we see that
\begin{multline*}
\vol_{m-1} (\partial TK)\le \lim_{\d\to 0^+}\frac{\vol_m\big(T(K+\d\rho B^m)\big)-\vol_m(TK)}{\d}\\\le\det(T)\lim_{\d\to 0^+}\frac{\vol_m(K+\d\rho B^m)-\vol_m(K)}{\d}\stackrel{\eqref{eq:surface formula limit}}{=}\det(T) \vol_{m-1}(\partial K)\rho=\frac{\vol_m(TK)}{\vol_m(K)} \vol_{m-1}(\partial K)\rho.\tag*{\qedhere}
\end{multline*}
\end{proof}

\begin{remark} We stated Lemma~\ref{lem:properties of ratio} with $K$ being a convex body since that is all that we need herein. However, the proof does not rely on its convexity  in an essential way; all that is needed is that $K$ is a body in $\R^m$ whose boundary is sufficiently regular so that the identity~\eqref{eq:surface formula limit} holds (with $n$ replaced by $m$).
\end{remark}

Any matrix   $\sfB$ as in Lemma~\ref{lem:inductive step} must have a row with at least $n/m$ nonzero entries. Indeed, otherwise the total number of nonzero entries of $\sfB$ would be less than $m(n/m)=n$, so at least one of the $n$ columns  $\sfB$ would have to vanish, in contradiction to the assumed linear independence (as $s\ge 1$). Thus, there exists $j\in \m$ such that at least $\lceil n/m\rceil$ of the entries of $\sfB^*e_j\in \R^n$ do not vanish. Those entries are integers, so $\|\sfB^*e_j\|_{\ell_2^n}\ge \sqrt{\lceil n/m\rceil}$. Hence, the quantity $\|\sfB\|_{\ell_2^n\to \ell_2^m}=\|\sfB^*\|_{\ell_2^m\to \ell_2^n}$ in~\eqref{eq:recursive inequality} cannot be less than $\sqrt{\lceil n/m\rceil}$.

\begin{question}\label{Q:norm of integer matrix} Given $m,n\in \N$ and $C>1$, what is the order of magnitude of the largest $s=s(m,n,C)\in \N$ for which  there exists $\sfB\in \M_{m\times n}(\Z)$ such that any $s$ of the columns of $\sfB$ are linearly independent and
$$
\|\sfB\|_{\ell_2^n\to \ell_2^m}\le  C\sqrt{\frac{n}{m}}.
$$
\end{question}

The following lemma is a step towards Question~\ref{Q:norm of integer matrix}  that we will use in the implementation of Lemma~\ref{lem:inductive step}:

\begin{lemma}\label{lem:the matrix that we need}Suppose that $m,n\in \N$ satisfy $4\le m\le n$ and $n\ge (m\log m)/4$. There exist $s\in \N$ with $s\gtrsim m^2/n$ and $\sfB\in \M_{m\times n}(\Z)$ of rank $m$ such that any $s$ of the columns of $\sfB$ are linearly independent and
\begin{equation*}
\|\sfB\|_{\ell_2^n\to \ell_2^m}\lesssim\sqrt{\frac{n}{m}}.
\end{equation*}
\end{lemma}

Lemma~\ref{lem:the matrix that we need} suffices for our purposes, but it is not sharp. We will actually prove below that in the setting of Lemma~\ref{lem:the matrix that we need} for every $0<\e\le 1$ there exist $s\in \N$ with $s\gtrsim m^{1+\e}/n^\e=m(m/n)^\e\ge m^2/n$ and $\sfB\in \M_{m\times n}(\Z)$ of rank $m$ such that any $s$ of the columns of $\sfB$ are linearly independent and $\|\sfB\|_{\ell_2^n\to \ell_2^m}\lesssim_\e \sqrt{n/m}$.

While Question~\ref{Q:norm of integer matrix} arises naturally from Lemma~\ref{lem:inductive step}  and it is interesting in its own right, fully answering Question~\ref{Q:norm of integer matrix}  will not lead to removing the $o(1)$ term in Theorem~\ref{thm:main} altogether; the bottleneck in the ensuing reasoning that precludes obtaining such an answer to Question~\ref{Q:sharp} (if true) is elsewhere.

 \begin{proof}[Proof of Theorem~\ref{thm:main Q version} assuming Lemma~\ref{lem:inductive step}  and Lemma~\ref{lem:the matrix that we need}]  We will proceed by induction on $n$. In preparations for the base of the induction, we will first record the following estimate (which is sharp when the lattice is $\Z^n$). The Voronoi cell of a rank $n$ sublattice $\Lambda$ of $\Z^n$, namely the set
 $$K=\big\{x\in \R^n: \forall y\in \Lambda, \ \|x\|_{\ell_2^n}\le \|x-y\|_{\ell_2^n}\big\},
  $$
  is a $\Lambda$-parallelotope that satisfies $K\supseteq \frac12 B^n$. Indeed, if $y\in \Lambda\setminus \{0\}$, then $\|y\|_{\ell_2^n}\ge 1$ since $y\in \Z^n\setminus \{0\}$. Hence,
  $$\forall x\in \frac12 B^n, \qquad  \|x-y\|_{\ell_2^n}\ge \|y\|_{\ell_2^n}-\|x\|_{\ell_2^n}\ge \|x\|_{\ell_2^n}.$$
By Lemma~\ref{lem:inradius}, it follows that $\vol_{n-1}(\partial K)/\vol_n(K)\le 2n$. This gives the (weak) a priori bound $Q(n)\le 2n$.

Fix $n\in \N$ and suppose that there exists $m\in \N$ satisfying $4\le m\le n$ and $n\ge (m\log m)/4$. By using Lemma~\ref{lem:inductive step}  with the matrix $\sfB$ from Lemma~\ref{lem:the matrix that we need} we see that there is a universal constant $\kappa\ge 4$  for which
\begin{equation}\label{eq:Q recursive}
Q(n)\le \kappa \left(\frac{n^{\frac32}}{m}+Q(m)\sqrt{\frac{n}{m}}\right).
\end{equation}
We will prove by induction on $n\in \N$ the following upper bound on $Q(n)$, thus proving Theorem~\ref{thm:main Q version}:
\begin{equation}\label{eq:the bound with kappa}
 Q(n)\le 4\kappa \sqrt{n} e^{\sqrt{2(\log n)\log (2\kappa)}}.
\end{equation}

If $n\le 4\kappa^2$, then by the above discussion $Q(n)\le 2n\le 4\kappa\sqrt{n}$, so that~\eqref{eq:the bound with kappa} holds. If $n>4\kappa^2$, then define
\begin{equation}\label{eq:choose m}
 m\eqdef  \left\lfloor ne^{-\sqrt{2(\log n)\log (2\kappa)}}\right\rfloor.
\end{equation}
It is straightforward to verify that this choice of $m$ satisfies $4\le m< n$ and $n\ge (m\log m)/4$ (with room to spare). Therefore~\eqref{eq:Q recursive} holds. Using the induction hypothesis, it follows that
\begin{align}\label{eq:substitute our m}
\begin{split}
Q(m)\sqrt{\frac{n}{m}}\le 4\kappa \sqrt{n} e^{\sqrt{2(\log m)\log (2\kappa)}}&\stackrel{\eqref{eq:choose m}}{\le} 4\kappa \sqrt{n} e^{\sqrt{2\left(\log n-\sqrt{2(\log n)\log (2\kappa)}\right)\log (2\kappa)}}\\&\le 4\kappa\sqrt{n} e^{\left(\sqrt{2\log n}-\sqrt{\log (2\kappa)}\right)\sqrt{\log (2\kappa)}}=2\sqrt{n} e^{\sqrt{2(\log n)\log (2\kappa)}},
\end{split}
\end{align}
where the penultimate step of~\eqref{eq:substitute our m} uses the  inequality $\sqrt{a-b}\le \sqrt{a}-b/(2\sqrt{a})$, which holds for  every $a,b\in \R$ with $a\ge b$; in our setting $a=\log n$ and $b=\sqrt{2(\log n)\log (2\kappa)}$ and $a>b$ because we are now treating the case $n>4\kappa^2$. A substitution of~\eqref{eq:substitute our m} into~\eqref{eq:Q recursive},  while using  that $m\ge \frac12 n\exp\left(-\sqrt{2(\log n)\log (2\kappa)}\right)$ holds thanks to~\eqref{eq:choose m}, gives~\eqref{eq:the bound with kappa}, thus completing the proof of Theorem~\ref{thm:main Q version}.
\end{proof}

We will next prove Lemma~\ref{lem:inductive step}, which is the key recursive step that underlies Theorem~\ref{thm:main}.

\begin{proof}[Proof of Lemma~\ref{lem:inductive step}]  We will start with the following two elementary observations to facilitate the ensuing proof. Denote the span of the rows of $\sfB$ by $V=\sfB^*\R^m\subset \R^n$ and notice that $\dim(V)=m$ as $\sfB$ is assumed to have rank $m$. Suppose that $\Lambda$ is a lattice of rank $n$ that is contained in $\Z^n$.
Firstly, we claim that the rank of the lattice $V^\perp \cap \Lambda$ equals $n-m$.
Indeed, we can write $V^\perp \cap \Lambda=\sfC(\Z^n\cap \sfC^{-1}V^\perp)$ where $\sfC$ is an invertible matrix with integer entries, i.e., $
\sfC \in \M_n(\Z)\cap \GL_n(\Q)$, such that $\Lambda =\sfC\Z^n$.
Furthermore, $V^\perp=\mathrm{Ker}(\sfB)$, so the dimension over $\Q$ of $\Q^n\cap V^\perp$ equals $n-m$. As $\sfC^{-1}\in \GL_n(\Q)$, it follows that $\sfC^{-1}V^\perp$ contains $n-m$ linearly independent elements of $\Z^n$.
Secondly, we claim that the orthogonal projection $\proj_V\Lambda$ of $\Lambda$ onto $V$ is a discrete subset of $V$, and hence is a lattice; its rank will then be  $\dim(V)=m$ because we are assuming that $\spn(\Lambda)=\R^n$, so $\spn(\proj_V\Lambda)=\proj_V(\spn(\Lambda))=\proj_V(\R^n)=V$. We need to check that for any  $\{x_1,x_2,\ldots\}\subset \Lambda$ such that $\lim_{i\to\infty} \proj_V x_i=\0$ there is $i_0\in \N$ such that $\proj_V x_i=\0$ whenever $i\in \{i_0,i_0+1,\ldots\}$. Indeed, as $V^\perp=\mathrm{Ker}(\sfB)$ we have $\sfB x=\sfB \proj_V x$ for every $x\in \R^n$, so $\lim_{i\to \infty}\sfB x_i=\0$. But, $\sfB x_i\in \Z^m$ for every $i\in \N$ because $\sfB\in \M_{m\times n}(\Z)$ and $x_i\in \Lambda\subset \Z^n$. Consequently, there is $i_0\in \N$ such that $\sfB x_i=\0$ for every $i\in \{i_0,i_0+1,\ldots\}$, i.e., $x_i\in \mathrm{Ker}(\sfB)=V^\perp$ and hence $\proj_V x_i=\0$.

Let $K_1\subset V^\perp$ be the Voronoi cell of $V^\perp \cap \Lambda$, namely
$
K_1=\{x\in V^\perp:\ \forall y\in V^\perp \cap \Lambda,\quad  \|x\|_{\ell_2^n}\le \|x-y\|_{\ell_2^n}\}$. 
If $y=(y_1,\ldots,y_n)\in V^\perp=\mathrm{Ker}(\sfB)$, then $y_1\sfB e_1+\cdots+y_n\sfB e_n=\0$. By the assumption on $\sfB$, this implies that if also $y\neq \0$, then $|\{i\in \n:\ y_i\neq 0\}|>s$. Consequently, as the entries of elements of $\Lambda$ are integers,
 $$
\forall y\in (V^\perp \cap \Lambda)\setminus \{0\}, \qquad\|y\|_{\ell_2^n}> \sqrt{s}.
$$
Hence, if $x\in \frac{\sqrt{s}}{2}(V^\perp \cap B^n)$, then
$$
\forall y\in (V^\perp \cap \Lambda)\setminus \{0\},\qquad \|x-y\|_{\ell_2^n}\ge \|y\|_{\ell_2^n}-\|x\|_{\ell_2^n}>\sqrt{s}-\frac{\sqrt{s}}{2}=\frac{\sqrt{s}}{2}\ge \|x\|_{\ell_2^n}.
$$
This means that $K_1\supseteq \frac{\sqrt{s}}{2} (V^\perp \cap B^n)$, and therefore by Lemma~\ref{lem:inradius}  we have
\begin{equation}\label{eq:K1 ratio}
\frac{\vol_{n-m-1}(\partial K_1)}{\vol_{n-m}(K_1)}\le \frac{n-m}{\frac12 \sqrt{s}}= \frac{2(n-m)}{\sqrt{s}}.
\end{equation}

Next, fix  $i\in \m$. By the definition of $V$,  the $i$'th row $\sfB^* e_i$ of $\sfB$ belongs to $V$, so
\begin{equation}\label{eq:dual rows in V}
\forall (x,i)\in \R^n\times \m,\qquad \langle x, \sfB^* e_i\rangle = \langle \proj_V x, \sfB^* e_i\rangle.
\end{equation}
Since all of the entries of $\sfB$  are integers, it follows that
$$
\forall (x,i)\in \Z^n\times \m,\qquad \langle \sfB \proj_V x,e_i\rangle= \langle  \proj_V x,\sfB^*e_i\rangle\stackrel{\eqref{eq:dual rows in V}}{=} \langle  x,\sfB^*e_i\rangle\in \Z.
$$
In other words, $\sfB\proj_V\Z^n\subset \Z^m$, and hence the lattice $\sfB\proj_V\Lambda$ is a subset of $\Z^m$.  Furthermore, $\sfB$ is injective on $V$ because $\mathrm{Ker}(\sfB)=V^\perp$, so $\sfB\proj_V\Z^n$ is a rank $m$ sublattice of $\Z^m$. By the definition of $Q(m)$, it follows that there exists a $\sfB\proj_V\Lambda$-parallelotope $K_2^0\subset \R^m$ such that
\begin{equation}\label{eq:use Qm}
\frac{\vol_{m-1}(\partial K_2^0)}{\vol_m(K_2^0)}\le Q(m).
\end{equation}

Because $V^\perp=\mathrm{Ker}(\sfB)$ and the rank of $\sfB$ is $m=\dim(V)$, the restriction $\sfB|_V$ of $\sfB$ to $V$ is an isomorphism between $V$ and $\R^m$. Letting $T:\R^m\to V$  denote the inverse of $\sfB|_V$, define
$
K_2= TK_2^0.
$
By combining (the second part of) Lemma~\ref{lem:properties of ratio} with~\eqref{eq:use Qm}, we see that 
\begin{equation}\label{eq:ratio for K2}
\frac{\vol_{m-1}(\partial K_2)}{\vol_m(K_2)}\le Q(m)\|\sfB\|_{\ell_2^n\to \ell_2^m}.
\end{equation}

Let $K= K_1+K_2\subset \R^n$. By combining (the first part of) Lemma~\ref{lem:properties of ratio} with~\eqref{eq:K1 ratio} and~\eqref{eq:ratio for K2}, we have
$$
\frac{\vol_{n-1}(\partial K)}{\vol_n(K)}\le  \frac{2(n-m)}{\sqrt{s}}+Q(m)\|\sfB\|_{\ell_2^n\to \ell_2^m}.
$$
Hence, the proof of Lemma~\ref{lem:inductive step} will be  complete if we check that $K$ is a $\Lambda$-parallelotope. Our construction ensures by design that this is so, as $K_1$ is a $(V^\perp \cap \Lambda)$-parallelotope and $K_2$ is a $\proj_V\Lambda$-parallelotope; verifying this fact is merely an unravelling of the definitions, which we will next perform for completeness.

Fix $z\in \R^n$. As $\R^m=\sfB\proj_V\Lambda+K_2^0$, there is $x\in \Lambda$ with $\sfB\proj_V z\in \sfB\proj_V x+K_2^0$. Apply $T$ to this inclusion and use that $T\sfB|_V$ is the identity mapping to get  $\proj_V z\in \proj_V x+K_2$. Next, $V^\perp=K_1+ V^\perp \cap \Lambda$ since $K_1$ is the Voronoi cell of $V^\perp \cap \Lambda$, so there is $y\in V^\perp \cap \Lambda$ such that $\proj_{V^\perp} z-\proj_{V^\perp} x\in y+K_1$. Consequently, $
z=\proj_{V^\perp} z+\proj_Vz\in \proj_{V^\perp} x+y+K_1+\proj_V x+K_2= x+y+K\in \Lambda+K$. Hence, $\Lambda+K=\R^n$.

It remains to check that for every $w\in \Lambda\setminus \{0\}$ the interior of $K$ does not intersect  $w+K$. Indeed, by the definition of $K$,  if $k$ belongs to the interior of $K$, then $k=k_1+k_2$, where $k_1$ belongs to the interior of $K_1$ and $k_2$ belongs to the interior of $K_2$. Since $\sfB$ is injective on $K_2\subset V$, it follows that $\sfB k_2$ belongs to the interior of $\sfB K_2=K_2^0$. If $\proj_{V}w\neq 0$, then $\sfB \proj_Vw\in \sfB\proj_V\Lambda\setminus \{0\}$, so because $K_2^0$ is a $\sfB\proj_V\Lambda$-parallelotope, $\sfB k_2\notin \sfB \proj_Vw +K_2^0$. By applying $T$ to is inclusion, we see that $k_2\notin \proj_Vw+ K_2$, which implies that $k\notin w+K$. On the other hand, if $\proj_{V}w=0$, then $w\in (V^\perp\cap \Lambda)\setminus \{0\}$. Since $K_1$ is a $V^\perp \cap \Lambda$-parallelotope, it follows that $k_1\notin w+K_1$, so $k\notin w+K$.
\end{proof}

To complete the proof of Theorem~\ref{thm:main Q version}, it remains to prove Lemma~\ref{lem:the matrix that we need}. For ease of later reference, we first record the following straightforward linear-algebraic fact:

\begin{observation}\label{obs:linear algebra} Fix $m,n,s\in \N$ with $s\le m\le n$. Suppose that there exists $\sfA\in \M_{m\times n}(\Z)$ such that any $s$ of the columns of $\sfA$ are linearly independent. Then, there also exists $\sfB\in \M_{m\times n}(\Z)$ such that any $s$ of the columns of $\sfB$ are linearly independent, $\sfB$ has rank $m$, and
\begin{equation}\label{eq:B norm}
\|\sfB\|_{\ell_2^n\to \ell_2^m} \le \sqrt{1+ \|\sfA\|_{\ell_2^n\to \ell_2^m}^2}.
\end{equation}
\end{observation}

\begin{proof} Let $r\in \m$ be the rank of $\sfA$. By permuting the rows of $\sfA$, we may assume that its first $r$ rows, namely $\sfA^*e_1,\ldots,\sfA^*e_r\in \R^n$ are linearly independent. Also, since we can complete $\sfA^*e_1,\ldots,\sfA^*e_r$ to a basis of $\R^n$ by adding $n-r$ vectors from $\{e_1,\ldots,e_n\}\subset \R^n$, by permuting the columns of $\sfA$, we may assume that the vectors $\sfA^*e_1,\ldots,\sfA^*e_r, e_{r+1},\ldots,e_{m}\in \R^n$ are linearly independent. Let $\sfB\in \M_{m\times n}(\Z)$ be the matrix whose rows are $\sfA^*e_1,\ldots,\sfA^*e_r, e_{r+1},\ldots,e_{m}$, so that $\sfB$ has rank $m$ by design. Also,
$$
\forall x\in \R^n,\qquad \|\sfB x\|_{\ell_2^m}^2=\sum_{i=1}^r (\sfA x)_i^2+\sum_{j=r+1}^{m} x_j^2\le \big(\|\sfA\|_{\ell_2^n\to \ell_2^m}^2+1\big)\|x\|_{\ell_2^n}^2.
$$
Therefore~\eqref{eq:B norm} holds. It remains to check that any $s$ of the columns of $\sfB$ are linearly independent. Indeed, fix $S\subset \n$ with $|S|=s$ and $\{\alpha_j\}_{j\in S} \subset \R$ such that $\sum_{j\in S} \alpha_j \sfB_{ij} =0$ for every $i\in \m$. In particular, $\sum_{j\in S} \alpha_j \sfA_{ij} =0$ for every $i\in \{1,\ldots,r\}$. If $k\in \{r+1,\ldots,m\}$, then since the $k$'th row of $\sfA$ is in the span of the first $r$ rows of $\sfA$, there exist $\beta_{k1},\ldots,\beta_{kr}\in \R$ such that $\sfA_{kj} = \sum_{i=1}^r \beta_{ki} \sfA_{ij}$ for every $j\in \n$. Consequently, $\sum_{j\in S} \alpha_j \sfA_{kj} =\sum_{i=1}^r \beta_{ki} \sum_{j\in S}\alpha_j\sfA_{ij}=0$.  This shows that $\sum_{j\in S} \alpha_j \sfA_{ij} =0$ for every $i\in \m$. By the assumed property of $\sfA$, this implies that $\alpha_j=0$ for every $j\in S$.
\end{proof}

The following lemma is the main existential statement that underlies our justification of Lemma~\ref{lem:the matrix that we need}:

\begin{lemma}\label{lem:LDPC} There exists a universal constant $c>0$ with the following property. Let $d,m,n\ge 3$ be integers that satisfy $d\le m\le n$ and $n\ge (m\log m)/d$. Suppose also that $s\in \N$ satisfies
\begin{equation}\label{eq:s assumption}
s\le \frac{c}{d}\left(\frac{m^d}{n^2}\right)^{\frac{1}{d-2}}.
\end{equation}
Then, there exists an $m$-by-$n$ matrix $\sfA\in M_{m\times n}(\{0,1\})$ with the following properties:
\begin{itemize}
\item Any $s$ of the columns of $\sfA$ are linearly independent over the field $\Z/(2\Z)$;
\item Every column of $\sfA$ has at most $d$ nonzero entries;
\item Every row of $\sfA$ has at most $5dn/m$ nonzero entries.
\end{itemize}
\end{lemma}

The ensuing proof of Lemma~\ref{lem:LDPC} consists of probabilistic reasoning that is common in the literature on Low Density Parity Check (LDPC) codes; it essentially follows the seminal work~\cite{Gal63}. While similar considerations appeared in many places, we could not locate a reference that  states Lemma~\ref{lem:LDPC}.\footnote{The standard range of parameters that is discussed in the LDPC literature is, using the notation of Lemma~\ref{lem:LDPC}, either when $m\asymp n$, or when $s,d$ are fixed and the pertinent question becomes how large $n$ can be as $m\to \infty$; sharp bounds in the former case are due to~\cite{Gal63} and sharp bounds in the latter case are due to~\cite{LPS97,NV08}. Investigations of these issues when the parameters have  intermediate asymptotic behaviors appear in~\cite{FKO06,Fei08,AF09,DHLMNPRSV12, GKM22,HKM22}.} A peculiarity of the present work is that, for the  reason that we have seen in the above deduction of Theorem~\ref{thm:main Q version} from Lemma~\ref{lem:inductive step}  and Lemma~\ref{lem:the matrix that we need}, we need to choose a nonstandard dependence of $m$ on $n$; recall~\eqref{eq:choose m}.

In the course of the proof of Lemma~\ref{lem:LDPC} we will use the following probabilistic estimate:

\begin{lemma}\label{lem:random walk} Let $\{W(t)=(W(t,1),\ldots,W(t,m))\}_{t=0}^\infty$ be the standard random walk on the discrete hypercube $\{0,1\}^m$, starting at the origin. Thus, $W(0)=\0$ and for each $t\in \N$ the random vector  $W(t)$ is obtained from the random vector $W(t-1)$ by choosing an index $i\in \m$ uniformly at random and setting $$W(t)=\big(W(t-1,1),\ldots,W(t-1,i-1),1-W(t-1,i),W(t-1,i+1),\ldots,W(t-1,m)\big).$$
Then, $\Pr [W(t)= \0]\le 2(t/m)^{t/2}$ for every $t\in \N$.
\end{lemma}

\begin{proof} If $t$ is odd, then $\Pr [W(t)= \0]=0$, so suppose from now that $t$ is even. Let $\sfP\in \M_{\{0,1\}^m\times \{0,1\}^m}(\R)$ denote the transition matrix of the  random walk $W$, i.e.,
$$
\forall f:\{0,1\}^m\to \R,\ \forall x\in \{0,1\}^m,\qquad \sfP f(x)=\frac{1}{m}\sum_{i=1}^m f(x+e_i\bmod 2).
$$
Then, $\Pr [W(t)= \0]=(\sfP^{t})_{\0\0}$. By symmetry, all of the $2^m$ diagonal entries of $\sfP^{t}$ are equal to each other, so $(\sfP^{t})_{\0\0}=\trace(\sfP^{t})/2^m$.
For every $S\subset \{0,1\}^m$, the Walsh function $(x\in \{0,1\}^m)\mapsto (-1)^{\sum_{i\in S} x_i}$ is an eigenvector of $\sfP$ whose eigenvalue equals $1-2|S|/m$. Consequently,
\begin{equation}\label{eq:ED2}
\Pr [W(t)= \0]=\frac{1}{2^m}\trace(\sfP^{t})=\frac{1}{2^m}\sum_{k=0}^m \binom{m}{k}\bigg(1-\frac{2k}{m}\bigg)^{t}.
\end{equation}

Suppose that $\beta_1,\ldots,\beta_m$ are independent $\{0,1\}$-valued unbiased Bernoulli random variables, namely, $\Pr[\beta_i=0]=\Pr[\beta_i=1]=1/2$ for any $i\in \m$. By Hoeffding's inequality (e.g.,~\cite[Theorem~2.2.6]{Ver18}),
\begin{equation}\label{eq:hoef}
\forall u\ge 0,\qquad \Pr \Bigg[\bigg|\sum_{i=1}^m \Big(\beta_i-\frac12\Big)\bigg|\ge u\Bigg]\le 2e^{-\frac{2u^2}{m}}.
\end{equation}
Observing that the right hand side of~\eqref{eq:ED2} is equal to the expectation of $\big(1-\frac{2}{m}\sum_{i=1}^m\beta_i\big)^{t}$, we see that
\begin{multline*}
\Pr [W(t)= \0]\stackrel{\eqref{eq:ED2}}{=}\left(-\frac{2}{m}\right)^{t}\E\Bigg[\bigg(\sum_{i=1}^m \Big(\beta_i-\frac12\Big)\bigg)^{t}\Bigg]= \left(\frac{2}{m}\right)^{t}\int_0^\infty t u^{t-1}\Pr \Bigg[\bigg|\sum_{i=1}^m \Big(\beta_i-\frac12\Big)\bigg|\ge u\Bigg]\ud u\\\stackrel{\eqref{eq:hoef}}{\le} 2t\left(\frac{2}{m}\right)^{t}\int_0^\infty u^{t-1}e^{-\frac{2u^2}{m}} \ud u=2\left(\frac{2}{m}\right)^{\frac{t}{2}}\left(\frac{t}{2}\right)!\le 2\left(\frac{2}{m}\right)^{\frac{t}{2}}\left(\frac{t}{2}\right)^{\frac{t}{2}}=2\left(\frac{t}{m}\right)^{\frac{t}{2}}.\tag*{\qedhere}
\end{multline*}
\end{proof}

With Lemma~\ref{lem:random walk}  at hand, we can now prove Lemma~\ref{lem:LDPC}.

\begin{proof}[Proof of Lemma~\ref{lem:LDPC}]  Consider the random matrix $\sfA\in \M_{m\times n}(\{0,1\})$ whose columns are independent identically distributed copies $W_1(d),\ldots,W_n(d)$  of $W(d)$, where $W(0)=\0,W(1),W(2),\ldots$ is the standard random walk on $\{0,1\}^m$ as in Lemma~\ref{lem:random walk}.  By design, this means that each column of $\sfA$ has at most $d$ nonzero entries. Fixing $(i,j)\in \m\times \n$, if $W_j(d,i)=1$, then in at least one of the $d$ steps of the random walk that generated $W_j(d)$ the $i$th coordinate was changed.  The probability of the latter event equals $1-(1-1/m)^d$. Hence, $\Pr[W_j(d,i)=1]\le 1-(1-1/m)^d\le d/m$ and therefore for every fixed $S\subset \n$, the probability that $W_j(d,i)=1$ for every $j\in S$ is at most  $(d/m)^{|S|}$. Consequently, the probability that all of the rows of $\sfA$ have at most $\ell=\lceil 4dn/m\rceil$ nonzero entries is at least
$$
1-m\binom{n}{\ell} \left(\frac{d}{m}\right)^\ell\ge 1-m\left(\frac{en}{\ell}\right)^\ell \left(\frac{d}{m}\right)^\ell=1-m\left(\frac{edn}{ m\ell}\right)^\ell\ge 1-m\left(\frac{e}{4}\right)^{4\log m}\ge \frac13,
$$
where the first step is an application of Stirling's formula, the penultimate step uses $\ell\ge 4dn/m$ and the assumption $n\ge (m\log m)/d$, and the final step holds because $m\ge 3$.

It therefore suffices to prove that with probability greater than $2/3$ the vectors $\{W_i(d)\}_{i\in S}\subset \{0,1\}^m$ are linearly independent over $\Z/(2\Z)$ for every $\emptyset \neq S\subset \n$ with $|S|\le  s$, where $s\in \N$ satisfies~\eqref{eq:s assumption} and the universal constant $c>0$ that appears in~\eqref{eq:s assumption} will be specified later;  see~\eqref{eq:specify c}.
So, it suffices to prove that with probability greater than $2/3$  we have $\sum_{i\in S} W_i(d)\not\equiv \0\bmod 2$ for every $\emptyset\neq S\subset \n$ with $ |S|\le s$. Hence, letting $D$ denote the number of $\emptyset\neq S\subset \n$ with $ |S|\le s$ that satisfy $\sum_{i\in S} W_i(d)\equiv \0\bmod 2$, it  suffices to prove that $2/3 <\Pr[D=0]=1-\Pr[D\ge 1]$. Using Markov's inequality, it follows that the proof of Lemma~\ref{lem:LDPC} will be complete if we demonstrate that $\E[D]<1/3$.

The expectation of $D$ can be computed exactly. Indeed,
\begin{equation}\label{eq:ED1}
\E[D]=\E\bigg[\sum_{\substack{S\subset \n\\1\le |S|\le s}} \1_{\left\{\sum_{i\in S} W_i(d)\equiv \0\bmod 2\right\}}\bigg]=\sum_{r=1}^s \binom{n}{r}\Pr [W(dr)= \0],
\end{equation}
where we used the fact that $\sum_{i\in S} W_i(d)\bmod 2\in \{0,1\}^m$ has the same distribution as $W(d|S|)$ for every $\emptyset \neq S\subset \n$. By substituting the conclusion of Lemma~\ref{lem:random walk} into~\eqref{eq:ED1} we see that
\begin{equation}\label{eq:ED bound}
\E[D]\le2\sum_{r=1}^s \binom{n}{r}\left(\frac{dr}{m}\right)^{\frac{dr}{2}}\le  2\sum_{r=1}^s \bigg(\frac{e d^{\frac{d}{2}}r^{\frac{d}{2}-1}n}{m^{\frac{d}{2}}}\bigg)^r,
\end{equation}
where in the last step we bounded the binomial coefficient using Stirling's formula. For every $r\in \{1,\ldots,s\}$,
\begin{equation}\label{eq:for geometric}
\frac{ed^{\frac{d}{2}}r^{\frac{d}{2}-1}n}{m^{\frac{d}{2}}}\le  \frac{ed^{\frac{d}{2}}s^{\frac{d}{2}-1}n}{m^{\frac{d}{2}}}\stackrel{\eqref{eq:s assumption}}{\le}edc^{\frac{d}{2}-1}<\frac17,
\end{equation}
provided that
\begin{equation}\label{eq:specify c}
c<\inf_{d\ge 3} \left(\frac{1}{7ed}\right)^{\frac{2}{d-2}}\in (0,1).
\end{equation}
Therefore, when~\eqref{eq:specify c} holds we may substitute~\eqref{eq:for geometric} into~\eqref{eq:ED bound} to get that $\E[D]<2\sum_{r=1}^\infty \frac{1}{7^{r}}=\frac13$.
\end{proof}

We can now prove Lemma~\ref{lem:the matrix that we need}, thus concluding the proof of Theorem~\ref{thm:main Q version}.

\begin{proof}[Proof of Lemma~\ref{lem:the matrix that we need}] We will prove the following stronger statement (Lemma~\ref{lem:the matrix that we need} is its special case $\e=1$). If $0<\e\le 2$ and $m,n\in \N$ satisfy $2+\lfloor 2/\e\rfloor\le m\le n$ and $n\ge (m\log m)/(2+\lfloor 2/\e\rfloor)$, then there exist $s\in \N$ with $s\gtrsim \e m^{1+\e}/n^\e$, and $\sfB\in \M_{m\times n}(\Z)$ such that any $s$ of the columns of $\sfB$ are linearly independent, the rows of $\sfB$ are linearly independent, and
\begin{equation*}
\|\sfB\|_{\ell_2^n\to \ell_2^m}\lesssim\frac{1}{\e}\sqrt{\frac{n}{m}}.
\end{equation*}

Indeed, apply Lemma~\ref{lem:LDPC} with $d=2+\lfloor 2/\e\rfloor\ge 3$ (equivalently, $d\ge 3$ is the largest integer such that $2/(d-2)\ge \e$) to deduce that there exist an integer $s$ with
$$
s\asymp \frac{1}{d}\left(\frac{m^d}{n^2}\right)^{\frac{1}{d-2}}=\frac{m}{d} \left(\frac{m}{n}\right)^{\frac{2}{d-2}}\asymp \e m \left(\frac{m}{n}\right)^{\e}=\frac{\e m^{1+\e}}{n^\e},
$$
and a matrix  $\sfA\in \M_{m\times n}(\{0,1\})\subset \M_{m\times n}(\Z)$ such that any $s$ of the columns of $\sfA$ are linearly independent over $\Z/(2\Z)$, every column of $\sfA$ has at most $d$ nonzero entries, and every row of $\sfA$ has at most $5dn/m$ nonzero entries. If a set of vectors $v_1,\ldots,v_s\in \{0,1\}^m$ is linearly independent over  $\Z/(2\Z)$, then it is also linearly independent over $\R$ (e.g., letting $\mathsf{V}\in \M_{m\times s}(\{0,1\})$ denote the matrix whose columns are  $v_1,\ldots,v_s$, the latter requirement is equivalent to the determinant of $\mathsf{V}^*\mathsf{V}\in \M_s(\{0,1\})$ being an odd integer, so in particular it does not vanish). Hence, any $s$ of the columns of $\sfA$ are linearly independent over $\R$. Also,
\begin{equation*}\label{eq:bound op by row column}
\|\sfA\|_{\ell_2^n\to \ell_2^m}\le \Big(\max_{i\in \m} \sum_{j=1}^n |\sfA_{ij}|\Big)^\frac12\Big(\max_{j\in \n} \sum_{i=1}^m |\sfA_{ij}|\Big)^\frac12\le \sqrt{\frac{5dn}{m}}\cdot\sqrt{d}\asymp \frac{1}{\e}\sqrt{\frac{n}{m}},
\end{equation*}
where the first step is a standard bound  which holds for any $m$-by-$n$ real matrix (e.g.~\cite[Corollary~2.3.2]{GvL13}). Thus, $\sfA$ has all of the properties that we require  from the  matrix $\sfB$ in Lemma~\ref{lem:the matrix that we need}, except that we do not know that $\sfA$ has rank $m$, but Observation~\ref{obs:linear algebra} remedies this (minor) issue.
\end{proof}

We end by asking the following question:

\begin{question}\label{Q:maxproj} Fix $n\in \N$. Does there exist an integer parallelotope $K\subset \R^n$ such that the $(n-1)$-dimensional area of the orthogonal projection $\proj_{\theta^\perp}K$ of $K$ along any direction $\theta\in S^{n-1}$ is at most $n^{o(1)}$?
\end{question}

An application of Cauchy's surface area formula (see~\cite[Section~5.5]{KR97}), as noted in, e.g.,~\cite[Section~1.6]{Nao21}, shows that a positive answer to Question~\ref{Q:maxproj} would imply Theorem~\ref{thm:main}. Correspondingly, a positive answer to Question~\ref{Q:maxproj} with $n^{o(1)}$ replaced by $O(1)$ would imply a positive answer to Question~\ref{Q:sharp}.

Apart from the  intrinsic geometric interest of Question~\ref{Q:maxproj},  if it had a positive answer, then we would deduce using~\cite{Nao21} that there exists an integer parallelotope $K\subset \R^n$ such that the normed space $\X$ whose unit ball is $K$ has certain desirable nonlinear properties, namely, we would obtain an improved randomized clustering of $\X$ and an improved extension theorem for Lipschitz functions on subsets of $\X$; we refer to~\cite{Nao21}  for the relevant formulations since including them here would result in a substantial digression.

\bibliographystyle{amsplain0}
\bibliography{parallelotope.bib}

 \end{document}